\documentclass{article}
\usepackage{amssymb}
\usepackage{amsmath}
\usepackage{amsthm}
\usepackage{enumerate}
\usepackage{framed}
\usepackage{braket}
\usepackage{bm}
\usepackage{mathrsfs}
\usepackage{accents}
\usepackage{tocloft}
\usepackage{graphicx}
\usepackage{tikz-cd}
\usetikzlibrary{patterns}
\usepackage{url}
\usepackage{color}
\usepackage{xifthen}
\usepackage{xcolor}
\usepackage{mdframed}
\usepackage{mathtools}
\usepackage[explicit]{titlesec}
\usepackage{geometry}
\usepackage{cleveref}
\usepackage{comment}
\usepackage{marvosym}

\renewcommand\emptyset{\varnothing}

\theoremstyle{definition}

\newtheorem{dfn}{Definition}
\crefname{dfn}{Definition}{Definitions}
\newtheorem{thm}{Theorem}
\crefname{thm}{Theorem}{Theorems}
\newtheorem{lem}{Lemma}
\crefname{lem}{Lemma}{Lemmas}

\crefname{cor}{Corollary}{Corollaries}
\newtheorem{prop}{Proposition}
\crefname{prop}{Proposition}{Propositions}
\newtheorem{question}{Question}
\crefname{question}{Question}{Questions}

\crefname{mlem}{Main Lemma}{Lemmas}

\crefname{nt}{Note}{Notes}

\crefname{exm}{Example}{Examples}
\crefname{figure}{Figure}{Figures}

\crefname{problem}{Problem}{Problems}
\crefname{equation}{Assertion}{Assertions}
\crefname{item}{}{}

\usepackage{amsmath}
\usepackage{cleveref}
\usepackage{enumitem}
\setlist[enumerate]{align=left}
\usepackage{tikz-cd}

\newlist{condenum}{enumerate}{3}
\setlist[condenum,1]{label=\textbf{\arabic*.},
                   ref  =\textbf{\arabic*.}}
\setlist[condenum,2]{label=\textbf{(\alph*)},
                   ref  =\themyenumi\textbf{(\alph*)}}
\setlist[condenum,3]{label=\bfseries(\roman*),
                   ref  =\themyenumii\textbf{.(\roman*)}}

\crefname{condenumi}{condition}{conditions}
\crefname{condenumii}{condition}{conditions}
\crefname{condenumiii}{condition}{conditions}

\DeclareMathOperator{\supp}{supp}
\DeclareMathOperator{\cof}{cf}

\DeclareMathOperator{\dom}{dom}

\newcommand{\limo}{\mathrm{Lim}}

\newcommand{\zfc}{\mathrm{ZFC}}

\newcommand{\pair}[1]{( {#1} )}

\newcommand{\seq}[2]{( #1 \mid #2 )}
\newcommand{\mdl}{\mathbb}
\newcommand{\restrict}{{\upharpoonright}}

\newcommand{\Mid}{\mathrel{}\middle|\mathrel{}}

\newcommand{\defset}[2]{\left\{ {#1} \Mid \text{#2} \right\}}

\newcommand{\ch}{\mathrm{CH}}

\newcommand{\ma}{\mathrm{MA}}

\newcommand{\forcestate}[2][]{\Vdash_{#1} \text{``{#2}"}}
\newcommand{\forces}[2][]{\Vdash_{#1} #2}

\renewcommand{\c}{\check}
\renewcommand{\d}{\dot}

\newcommand{\club}{\mathsf{club}}
\newcommand{\stat}{\mathsf{stat}}
\newcommand{\ubdd}{\mathsf{ubdd}}

\newcommand{\height}{\mathrm{ht}}

\newcommand{\finantich}{a}
\newcommand{\specialize}{\mathbb{S}}

\newcommand{\sigc}{\text{$\sigma$-centered}}
\newcommand{\sigl}{\text{$\sigma$-linked}}

\newcommand{\knas}{\mathrm{(K)}}
\newcommand{\xknas}[1]{{#1}\text{-(K)}}

\newcommand{\xpc}[1]{{#1}\text{-pc}}

\definecolor{reasontext}{rgb}{0,0,0}
\definecolor{reasonbg}{rgb}{0.9,0.9,0.9}

\title{Uniformization of ladder system colorings and stationary precaliber forcings}
\author{Yushiro Aoki\thanks{National Institute of Technology, Tokyo College Email: \protect\url{y_aoki@tokyo-ct.ac.jp}}}

\begin{document}
	\maketitle
	\sloppy
\begin{abstract}
	We investigate the relationship between variants of the uniformization property for ladder system colorings and fragments of Martin's Axiom. 
	The well-known forcing properties of having precaliber $\aleph_1$ and being $\sigma$-centered correspond to uncountable refinement 
	and countable decomposition into centered subsets, respectively, and the associated forcing axioms have been widely studied. 
	In this paper, we focus on a forcing axiom for the property corresponding to stationary refinement, namely the stationary precaliber $\aleph_1$ property. 
	Analogously, we observe that ladder system coloring uniformization also admits both stationary refinement and countable decomposition variants. We discuss the interaction between these uniformization properties and various forcing axioms. 
	Through this analysis, we obtain as a main result the separation between the forcing axioms for stationary precaliber $\aleph_1$ and for $\sigma$-linked posets.
\end{abstract}

\section{Introduction}
The study of separating fragments of Martin's Axiom has been a long-standing topic in set theory (e.g., see \cite{MR3418888}). 
The most typical fragments are forcing axioms of the form $\mathrm{MA}(\Phi)$, where $\Phi$ is a chain condition stronger than the ccc. 
A well-studied example of such a property is precaliber $\aleph_1$, which asserts that every uncountable set of conditions contains an uncountable centered subset. In other words, this property can be viewed as an uncountable refinement into centered sets. 
While the forcing axiom for precaliber $\aleph_1$ implies many consequences, here we focus on its implications for the (full) uniformization of ladder system colorings.

Ladder system coloring uniformization was introduced by Devlin and Shelah \cite{devlin-shelah} as a set-theoretic reformulation of the Whitehead problem. 
Specifically, the existence of a ladder system such that every coloring on a stationary set is uniformizable is equivalent to the existence of a non-free Whitehead group with $\aleph_1$-generators. 
Uniformization also appears as a consequence of Ramsey-type fragments of Martin's Axiom, which were introduced and studied by Todorčević \cite{Todorcevic1989PartitionPI}. 
Todor\v{c}evi\'{c}-Veli\v{c}kovi\'{c} showed that $\mathcal{K}'_4$ implies that every ladder system coloring is uniformizable on some club set (see \cite{yorioka2019todorvcevic}). 
Furthermore, Moore-Yorioka \cite{YoriokaMoore} proved that $\mathcal{K}'_2$ implies $\sigma\mathsf{U}$, that is, for every ladder system coloring, there exists a countable partition of $\omega_1$ such that the coloring is uniformizable on each piece.

Thus, combinatorial principles on $\omega_1$ often lead to partial uniformization properties, which are strictly weaker than full uniformization. 
On the other hand, the forcing axiom corresponding to precaliber~$\aleph_1$ implies full uniformization. 
To investigate the finer relationship between uniformization properties and forcing properties, it is necessary to introduce stronger forcing conditions.

From this viewpoint, we propose new variants of forcing properties and uniformizations. 
We consider properties stating that any sequence of conditions indexed by a stationary set has a centered (or linked) subsequence still indexed by a stationary set. 
These properties are referred to as stationary precaliber $\aleph_1$ and stationary Knaster, respectively.
Such notions have already been studied by Krueger in the context of forcing axioms for non-special Aronszajn trees \cite{KRUEGER2020102820}. 
In what follows, we use the notations $\xpc{\stat}$ and $\xknas{\stat}$ to denote these properties.

We also consider localized versions $\xpc{\stat_E}$ and $\xknas{\stat_E}$, where the stationary refinement is restricted to a fixed stationary set $E$. 
The trivial implications among forcing axioms for these properties that follow from their definitions are illustrated in \cref{fig:ma}.

\begin{figure}
	\centering
	\caption{Trivial implications of fragments of Martins axiom. $E_0 \cup E_1 = \omega_1$ are stationary sets.}
  \label{fig:ma}
\begin{tikzcd}[ampersand replacement=\&,row sep=.5cm,column sep=.2cm]
	\& 
		\& \ma_{\aleph_1}((\mathrm{K})) \ar[dll]\ar[d]\ar[dr]
			\& 
\\
\ma_{\aleph_1}(\text{PC}_{\aleph_1}) \ar[d]\ar[dr]
	\& 
		\& \ma_{\aleph_1}(\xknas{\stat_{E_0}}) \ar[dll]\ar[rd] \ar[dll]
			\& \ma_{\aleph_1}(\xknas{\stat_{E_1}})\ar[d]\ar[dll]
\\
 \ma_{\aleph_1}(\text{$\stat_{E_0}$-pc})\ar[rd] 
	\& \ma_{\aleph_1}(\text{$\stat_{E_1}$-pc})\ar[d]
		\& 
			\& \ma_{\aleph_1}(\xknas{\stat} )\ar[dll]\ar[d]
\\
	\& \ma_{\aleph_1}(\text{$\stat$-pc}) \ar[d]
		\& 
			\& \ma_{\aleph_1}(\sigl) \ar[dll]
\\
	\& \ma_{\aleph_1}(\sigc)
		\&
			\& 
\end{tikzcd}
\end{figure}
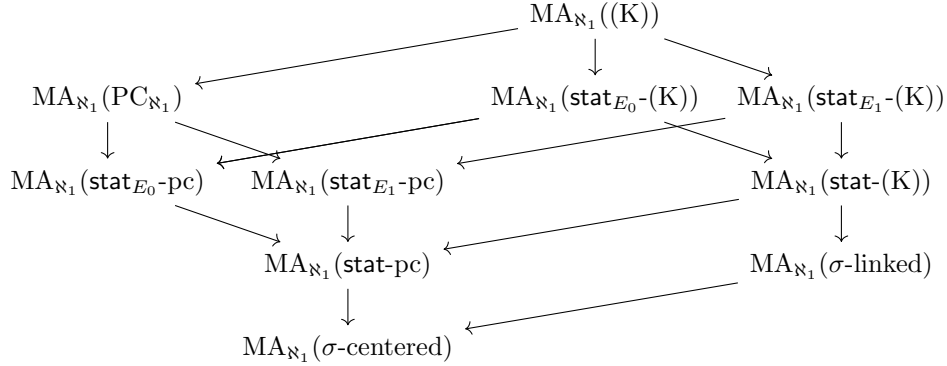
By analyzing the interaction between these forcing properties and variants of uniformization, we obtain separation results both among variants of uniformization principles and among fragments of forcing axioms.
\begin{enumerate}[label=(MA\arabic*)]
\item For a fixed partition $E_0 \cup E_1 = \omega_1$ into stationary sets, $\ma_{\aleph_1}(\xknas{\stat_{E_0}})$ does not imply $\ma_{\aleph_1} (\xpc{\stat_{E_1}})$ (\cref{thm:local sep}).
\item $\ma_{\aleph_1}(\xknas{\stat})$ and $\forall E \subset \omega_1 \text{:costationary},\,\lnot\ma_{\aleph_1}(\xpc{\stat_E})$ is consistent (\cref{thm:global sep}).
\item $\ma_{\aleph_1}(\sigl)$ does not imply $\ma_{\aleph_1}(\xpc{\stat})$ (\cref{thm:distinguish sigl and statpc}).
\end{enumerate}
 
For uniformizations,
\begin{enumerate}[label=(U\arabic*)]
	\item All ladder system colorings are uniformized on a fixed stationary set $E \subset \omega_1$ but a ladder system coloring is not uniformized on any stationary $E' \subset \omega_1 \setminus E$ (\cref{thm:local sep}).
	\item For every stationary set $E$, all ladder system colorings are uniformized on a stationary subset $E' \subset E$ but a ladder system coloring is not $\sigma$-uniformized, that is, for every decomposition $\bigcup_{n < \omega}A_n = \omega_1$, there exists $n$ such that the ladder system coloring is not uniformizable on $A_n$.
	\begin{itemize}
	\item More precisely, a ladder system coloring is not $\sigma$-$T$-uniformized for a fixed specail Aronszajn tree $T$.
	\end{itemize}
\end{enumerate}
Moreover, we see that $\ma_{\aleph_1}(\xknas{\stat})$ does not imply that any ladder system coloring is uniformized on some stationary set (\cref{thm:global sep}).

\section{Stationary precaliber properties}
We consider generalizations of well-known chain conditions.
\begin{dfn}
Let $\mathcal{A}$ be a set of subsets of $\omega_1$. 
$\mdl{P}$ is an \emph{$\mathcal{A}$-pc poset ($\mathcal{A}$-$n$-(K) poset)} if every sequence  of which index set is in $\mathcal{A}$ contains a centered ($n$-linked) subsequence of which index set is in $\mathcal{A}$.
An $\mathcal{A}$-(K) poset is a $\mathcal{A}$-$2$-(K) poset.
\end{dfn}
Note that properties pre-caliber $\aleph_1$ and $n$-Knaster are equivalent to $\ubdd$-pc and $\ubdd$-$n$-Knaster where $\ubdd$ denotes the family of unbounded subsets of $\omega_1$. 
Other families we consider is the following.

\begin{dfn}
For any stationary set $X$, let 
\begin{align*}
\stat_X &:= \{E \subset X \mid \text{$E$ is a stationary set in $\omega_1$} \}\\ 
\stat &:= \stat_{\omega_1}\\
\club_X &:= \{C \cap X \mid \text{$C$ is a club set in $\omega_1$} \}\\ 
\club &:= \club_{\omega_1}
\end{align*}
\end{dfn}

\begin{lem}\label{lem:statpc pres}
Let $E \in \stat$.
Every finite support iteration of $\stat_E$-pc posets is also $\stat_E$-pc.
\end{lem}
\begin{proof}
We first check the successor steps.
Let $\mdl{P}*\d{\mdl{Q}}$ be a two step iteration of $\stat_E$-precaliber posets.
Let $\seq{\pair{p_\alpha, \d{q}_\alpha}}{\alpha \in S}$ be a $S$-sequence in $\mdl{P} * \d{\mdl{Q}}$ where $S \in \stat_E$.
It is easy to see that some $p^* \in \mdl{P}$ forces $\d{T} := \{\alpha \mid p_\alpha \in \d{G}_{\mdl{P}} \} \in \stat_S$ since $\mdl{P}$ is ccc.
Let \[p^*\forcestate{$\d{T}_0 \in \stat_{\d{T}}$ and $\seq{\d{q}_\alpha}{\alpha \in \d{T}_0}$ is centered}\]
and let $T_1 = \{\alpha \in S \mid \exists p \leq p^*,\, p \leq p_\alpha \land p \forces{\c{\alpha} \in \d{T}_0} \}$.
Then $T_1 \in \stat_S$.
To see this, towards a contradiction, suppose that there exists $C \in \club$ such that $C \cap T_1 = \emptyset$.
Then, $p^* \forces{C \cap \d{T}_0 \neq \emptyset}$.
Pick $p \leq p^*$ and $\alpha \in C$ such that $p \forces{\alpha \in \d{T}_0} \subset \d{T}$.
Thus $p \forces{p_\alpha \in \d{G}_{\mdl{P}}}$ and hence $\alpha \in T_1$, it is a contradiction.

Select a sequence $\seq{p^+_\alpha \leq p_\alpha}{\alpha \in T_1}$ such that $p^+_\alpha \forces{\c{\alpha}\in\d{T}_0}$.
Pick $T_2 \in \stat_{T_1}$ such that $\seq{p^+_\alpha}{\alpha \in T_2}$ is centered.
Then $\seq{\pair{p_\alpha, \d{q}_\alpha}}{\alpha \in T_2}$ is centered.
Indeed, if $F \subset T_2$ is finite, 
then there exists $r \in \mdl{P}$ such that $\forall \alpha \in F,\,r \leq p^+_\alpha$.
Thus we have $r \forces{\c{F} \subset \d{T}_0}$ and hence $r \forces{\exists \d{q} \in \d{\mdl{Q}},\, \forall \alpha \in \c{F},\, \d{q} \leq \d{q}_\alpha}$.
A common extension of $\pair{p_\alpha, \d{q}_\alpha}$ $(\alpha \in F )$ exists.

We turn into the limit step.
Let $\gamma$ be a limit ordinal and let $\mdl{P}_\gamma = \seq{\mdl{P}_\xi, \d{\mdl{Q}}_\xi}{\xi < \gamma}$ be a fsi of $\stat_E$-precaliber posets.
We assume that every $\mdl{P}_\xi$ ($\xi < \gamma$) is an $\stat_E$-precaliber poset.
Fix an $S$-sequence $\seq{p_\alpha \in \mdl{P}_\gamma}{\alpha \in S}$ for some $S \in \stat_E$.
Fix an injection $\iota \colon \bigcup\{\supp p_\alpha \mid \alpha \in S\} \to \omega_1$ and let $a_\alpha = \iota^\to(\supp p_\alpha)$.
Using Fodor's pressing down lemma, pick $S_0 \in \stat_S$ and $n < \omega$ such that $\forall \alpha \in S_0,\,n = |a_\alpha \cap \alpha|$.
Again, using Fodor's pressing down lemma $n$ times,  pick $S_1 \in \stat_{S_0}$ and $a \in [\omega_1]^n$ such that $\forall \alpha \in S_1,\, a = a_\alpha \cap \alpha$.
Then $C = \{\delta \in \omega_1 \mid \forall \alpha \in \delta \cap S_1,\,a_\alpha \subset \delta \}$ is a club set.
Let $S_2 = S_1 \cap C \in \stat_E$.
Let $b = \iota^{\gets}(a)$.
Then, by the induction hypothesis, $\seq{p_\alpha \restrict b}{\alpha \in S_2}$ has a centered subsequence of which has the index set $T \in \stat_E$.
Then, $\seq{p_\alpha}{\alpha \in T}$ is centered.
\end{proof}
This proof also works for $\stat_E$-$n$-(K) posets too.
\begin{lem}
For $E \in \stat$ and $n < \omega$, every fsi of $\stat_E$-$n$-(K) posets is also $\stat_E$-$n$-(K)
\end{lem}
\begin{lem}
Let $\kappa$ be an uncountable cardinal.
The forcing axiom for $\xpc{\stat_E}$ ($\xknas{\stat_E}$) posets of size $\leq \kappa$ implies $\ma_{\kappa}(\xpc{\stat_E})$ ($\ma_{\kappa}(\xknas{\stat_E})$).
\end{lem}
\begin{proof}
	We may assume that the forcing axiom for $\xpc{\stat_E}$ ($\xknas{\stat_E}$) posets of size $\leq \kappa$.
	Let $\mdl{P}$ be a $\xpc{\stat_E}$ forcing notion.
	Let $\mathcal{D} = \seq{D_\alpha}{\alpha < \kappa}$ be a sequence of dense subsets of $\mdl{P}$.
	Let $\theta$ be a large enough regular cardinal and $M$ be an elementary submodel of $H_\theta$ of size $\kappa$ that contains $\kappa$ and enough elements to argue about $\mdl{P}$ and   $\mathcal{D}$.
	Then the poset $\mdl{P} \cap M$ of which order is induced by the one of $\mdl{P}$ is $\xpc{\stat_E}$.
	Indeed, if $\bar{p} = \seq{p_\alpha \in \mdl{P} \cap M}{\alpha \in E'}$ is a sequence of conditions in $M$ where $E' \in \stat_E$, then there exists $E'' \in \stat_{E'}$ such that $\bar{p} \restrict E''$ is centered in $\mdl{P}$.
 By elementarity, it is easily seen that $\bar{p} \restrict E''$ is centered in $\mdl{P} \cap M$, so is that each $D_\alpha \cap M$ is dense in $\mdl{P} \cap M$.
	Thus, by the assumption, there exists a $\seq{D_\alpha \cap M}{\alpha <\kappa}$-generic filter $G$.
	Let $\tilde{G} = \{p \in \mdl{P} \mid \exists q \in G,\, q \leq p\}$.
	Then $\tilde{G}$ is a $\mathcal{D}$-generic filter.
\end{proof}
For simplicity, we consider the following weakening of $\sigma$-linked and $\sigma$-centered, namely $\aleph_1$-$\sigma$-linked and $\aleph_1$-$\sigma$-centered, respectively.
\begin{dfn}
For a cardinal $\kappa$, $\mdl{P}$ is \emph{$\kappa$-$\sigl$ ($\kappa$-$\sigc$)} if every sequence $\seq{p_\alpha \in \mdl{P}}{\alpha < \kappa}$ of conditions indexed by $\kappa$ is $\sigl$ ($\sigc$) in $\mdl{P}$.
\end{dfn}
Obviously, $\mdl{P}$ is $\sigl$ if and only if $|\mdl{P}|$-$\sigl$.
Furtheremore, the forcing $\mdl{C}_\kappa$ adding $\kappa$ Cohen reals is $\aleph_1$-$\sigl$ since $\mdl{C}_{\omega_1}$ is $\aleph_1$-$\sigl$.
Thus some $\aleph_1$-$\sigl$ forcing grows the continuum.
It is obvious from the pressing down lemma that the property $\aleph_1$-$\sigl$ ($\aleph_1$-$\sigc$) is stronger than $\stat$-Knaster ($\xpc{\stat}$, respectively).
\begin{lem}
The properties $\aleph_1$-$\sigl$ and $\aleph_1$-$\sigc$ are closed under finite support iteration.
\end{lem}
Note that the properties $\sigl$ and $\sigc$ may not be closed under finite support iteration when their length are greater than the continuum.
\begin{proof}
We give the proof only for the linked case.
Proceed by the induction of the length of the iteration.

First, let us check the successor step.
Suppose that $\mdl{P}$ is $\aleph_1$-$\sigl$ and $\d{\mdl{Q}}$ is a $\mdl{P}$-name for an $\aleph_1$-$\sigl$ forcing.
Fix any sequence $\seq{\pair{p_\alpha,\d{q}_\alpha}}{\alpha < \omega_1}$ of conditions in $\mdl{P}*\d{\mdl{Q}}$.
Let $\forcestate[\mdl{P}]{$\{\d{q}_\alpha \mid \alpha < \omega_1\} = \bigcup_{n<\omega}\d{\Lambda}_n$ and each $\d{\Lambda}_n$ is linked}$.
For $\alpha < \omega_1$, select a pair of sequences $\pair{\bar{p}',\bar{n}} = \seq{p'_\alpha \leq p_\alpha, n_\alpha < \omega}{\alpha < \omega_1}$ such that $p'_\alpha \forces{\d{q}_\alpha \in \d{\Lambda}_{n_\alpha}}$ for each $\alpha < \omega_1$.
Let $\{p'_\alpha \mid \alpha < \omega_1\} = \bigcup_{n < \omega} \Lambda_n$ where each $\Lambda_n$ is linked.
Let $R_{m,n} = \{\pair{p_\alpha,\d{q}_\alpha} \mid p'_\alpha \in \Lambda_m \land n_\alpha = n \}$.
Then linked sets $R_{m,n}$ covers $\{\pair{p_\alpha,\d{q}_\alpha} \mid \alpha < \omega_1\}$.

Second, we proceed to the limit step.
Let $\mdl{P}_\gamma = *\seq{\mdl{P}_\alpha,\d{\mdl{Q}}_\alpha}{\alpha < \gamma}$ and fix an $\omega_1$-sequence $\seq{p_\alpha \in \mdl{P}_\gamma}{\alpha < \omega_1}$.
Proofs are clear for the cases $\cof(\gamma) = \omega$ and $\cof(\gamma) > \omega_1$.
For countable cofinality $\gamma$, select a cofinal increasing sequence $\seq{\gamma_n < \gamma}{n < \omega}$ in $\gamma$.
Then, defining $I_n = \{\alpha < \omega_1 \mid \supp(p_\alpha) \subset \gamma_n\}$, each $\seq{p_\alpha}{\alpha \in I_n}$ is covered by countable many linked sets.
For $\gamma$ of greater cofinality than $\omega_1$, $\{\max\supp(p_\alpha) \mid \alpha < \omega_1\}$ is bounded in $\gamma$.
So we prove the essential case $\cof(\gamma) = \omega_1$.
For simplicity of symbols, we may assume that $\gamma = \omega_1$.
By the induction hypothesis, let $\{p_\alpha \restrict (\xi + 1) \mid \alpha < \omega_1\} = \bigcup_{n < \omega}\Lambda^\xi_n$ where each $\Lambda^\xi_n$ is linked.
Fix an injection $g \colon \omega_1 \to {}^\omega2$.
For $h \in {}^{({}^N2)}\omega$, let \[\Lambda_h = \{p_\alpha \mid \forall \xi \in \supp(p_\alpha),\, p_\alpha \restrict (\xi + 1) \in \Lambda^\xi_{h(g(\alpha) \restrict N)}\}.\]
Then the family of $\Lambda_h$ ($h \in \bigcup_{N<\omega}{}^{({}^N2)}\omega$) covers $\{p_\alpha \mid \alpha < \omega_1\}$: For $\alpha < \omega_1$, select $N_\alpha < \omega$ such that $|\{g(\xi) \restrict N_\alpha \mid \xi \in \supp(p_\alpha)\}| = |\supp(p_\alpha)|$, i.e., $N_\alpha$-length sequences $g(\xi) \restrict N_\alpha$ ($\xi \in \supp(p_\alpha)$) are distinct from each other.
For $\xi \in \supp(p_\alpha)$, let $h^\alpha(g(\xi) \restrict N_\alpha) \in \omega$ be such that $p_\alpha \restrict(\xi + 1) \in \Lambda^\xi_{h^\alpha(g(\alpha)\restrict N_\alpha)}$.

Furthermore, each $\Lambda_h$ is linked:
For $p_\alpha,p_\beta \in \Lambda_h$, let $\xi = \max (\supp(p_\alpha) \cap \supp(p_\beta))$ if exists.
Then, since $p_\alpha \restrict(\xi + 1), p_\beta \restrict(\xi + 1) \in \Lambda^\xi_{h(g(\xi) \restrict N)}$ are compatible, select a common extension $q$ of them.
Then the condition $p$ defined by 
\[p(\eta) = \begin{cases}
q(\eta) & (\eta < \xi + 1) \\
p_\alpha(\eta) & (\eta \in \supp(p_\alpha) \setminus (\xi + 1)) \\
p_\beta(\eta) & (\text{otherwise})
\end{cases}\]
is a common extension of $p_\alpha$ and $p_\beta$.
Therefore, $\{\Lambda_h \mid h \in \bigcup_{N<\omega}{}^{({}^N2)}\omega \}$ is a countable family of linked sets that covers $\{p_\alpha \mid \alpha < \omega_1\}$.
\end{proof}
\begin{lem}
$\ma_{\aleph_1}(\sigl)$ and $\ma_{\aleph_1}(\sigc)$ are equivalent to $\ma_{\aleph_1}(\text{$\aleph_1$-}\sigl)$ and $\ma_{\aleph_1}(\text{$\aleph_1$-}\sigc)$, respectively.
\end{lem}
\begin{proof}
We give the proof only for the linked case.
The proof is similar to the standard one of the equivalence between $\ma_{\aleph_1}(\sigl)$ and $\ma_{\aleph_1}(\sigl + \text{``size $\leq \aleph_1$"})$.
Assume $\ma_{\aleph_1}(\sigl)$ and fix any $\aleph_1$-$\sigl$ poset $\mdl{P}$ and any sequence $\mathcal{D} = \seq{D_\alpha}{\alpha < \omega_1}$ of dense sets in $\mdl{P}$.
Let $\kappa$ be a large enough regular cardinal and select an elementary submodel $N \prec H_\kappa$ of size $\aleph_1$ which contains $\{\mathcal{D},\mdl{P}\} \cup \omega_1$.
Since $\mdl{P}$ is $\aleph_1$-$\sigl$, the poset $\mdl{P} \cap N$ is $\sigl$.
It is easily seen that each $D_\alpha \cap N$ is dense in $\mdl{P} \cap N$.
Thus a $\mathcal{D}\restrict N = \{D_\alpha \cap N \mid \alpha < \omega_1 \}$-generic filter $G \subset \mdl{P} \cap N$ exists by $\ma_{\aleph_1}(\sigl)$.
$G$ generates a $\mathcal{D}$-generic filter in $\mdl{P}$.
\end{proof}

\section{Ladder system coloring uniformization and $\stat_E$-pc}
We shall describe the relation between ladder system uniformizations and forcing axioms for properties $\stat_E$-pc or $\stat_E$-$\knas$.
\begin{dfn}
We say that $\vec{l} = \seq{l_\alpha \subset \alpha}{\alpha \in \omega_1 \cap \limo }$ is \emph{a ladder system} if each $l_\alpha$ is unbounded in $\alpha$ and of order type $\omega$.
For a ladder system $\vec{l}_\alpha$ on $S$ and $\nu \leq \omega$,
$\vec{c} = \seq{c_\alpha \colon l_\alpha \to \nu}{\alpha \in \omega_1 \cap \limo}$ is \emph{a $\nu$-coloring of $\vec{l}$}.
In this case, the pair $\pair{\vec{l}, \vec{c}}$ is called \emph{a ladder system coloring}.
Let $l_{\alpha,n}$ be the $n$-th element of $l_\alpha$ and $l_\alpha^{n} := l_\alpha \setminus l_{\alpha,n}$
A ladder system coloring $\pair{\vec{l},\vec{c}}$ is \emph{uniformized on $S$ over $k \colon S \to \omega$} if $\bigcup_{\alpha \in S} c_\alpha \restrict l^{k(\alpha)}_\alpha$ is a function.
For a family $\mathcal{A}$ of subsets of $\omega_1$, $\mathsf{U}(\mathcal{A})$ is the assertion that every ladder system coloring is uniformized on some $A \in \mathcal{A}$.
$\mathsf{U}$ denotes the assertion $\mathsf{U}(\{\omega_1\})$.
$\sigma\mathsf{U}$ is the assertion that, for every ladder system coloring $\pair{\vec{l},\vec{c}}$, there exists a partition $\omega_1 \cap \limo = \bigcup_{n < \omega}A_n$ such that $\pair{\vec{l},\vec{c}}$ is uniformized on each $A_n$.
\end{dfn}
Notice that $\sigma\mathsf{U}$ implies $\mathsf{U}(\stat_E)$ for all $E \in \stat$.
There are no defference between $\mathsf{U}(\club)$ and $\mathsf{U}$:
\begin{thm}\cite{Eklof:1992aa}
$\mathsf{U}(\club)$ implies (and hence is equivalent to) $\mathsf{U}$.
\end{thm}

\begin{dfn}
Let $\pair{\vec{l}, \vec{c}}$ be a ladder system coloring and $E$ be a stationary set.
\begin{align}
\mdl{P}_E(\vec{l},\vec{c}) = \defset{ p\in \mathrm{Fn}(E,\omega)}{ $\bigcup_{\alpha \in \dom(p)} c_\alpha \restrict l_\alpha^{p(\alpha)}$ is a function}
\end{align}
Here, $\mathrm{Fn}(X,Y)$ denotes the set of all finite partial functions $X \to Y$.
For $q,p \in \mdl{P}$, $q \leq p$ if $q \supset p$.
For $p \in \mdl{P}$, let $c^p = \bigcup_{\alpha \in \dom(p)} c_\alpha \restrict l_\alpha^{p(\alpha)}$.
Note that a finite set $F \subset \mdl{P}$ has a common extension if and only if both of $\bigcup F$ and $\bigcup_{p \in F} c^p$ are functions.
\end{dfn}
When $E$ is co-stationary, then the forcing notions of form $\mdl{P}_E(\vec{l},\vec{c})$ is very close to the Cohen forcing of all finite partial function $\omega_1 \to 2$ by the following lemma.
\begin{dfn}
	For a subset $X \subset \omega_1$, a forcing poset $\mdl{P}$ is \emph{semi-Cohen within $X$} if for every large enough regular cardinal $\kappa$, if $N \prec H_\kappa$ is countable such that $\omega_1 \cap N \in X$ and if $p \in \mdl{P}$, there exists $p^N \in \mdl{P} \cap N$ such that all conditions below $p^N$ in $N$ is compatible with $p$.
\end{dfn}
Thus the property semi-Cohen (see \cite{BALCAR1997187}) is equivalent to the property semi-Cohen within $\omega_1$ of our terminology.
\begin{prop}[\cite{Aoki1}]
$\mdl{P} = \mdl{P}_E(\vec{l},\vec{c})$ is \emph{semi-Cohen within $\omega_1 \setminus E$}, that is, for every large enough regular cardinal $\kappa$, if $N \prec H_\kappa$ is countable such that $\omega_1 \cap N \in \omega_1 \setminus E$ and if $p \in \mdl{P}$, there exists $p^N \in \mdl{P} \cap N$ such that all conditions below $p^N$ in $N$ is compatible with $p$.
\end{prop}
The property semi-Cohen within $\omega_1$ is equivalent to the proeprty semi-Cohen \cite{BALCAR1997187}.

\begin{lem}\label{ULC(E) is statEc}
Let $\pair{\vec{l},\vec{c}}$ be a ladder system coloring on a stationary-co-stationary set $E$.
$\mdl{P}_E(\vec{l},\vec{c})$ is $\stat_{E^c}$-precaliber poset. 
Here $E^c$ denotes $\omega_1 \setminus E$.
\end{lem}
\begin{proof}
Fix any $S \in \stat_{E^c}$.
Applying the pressing down lemma and cutting off with a club set, pick $S_1 \in \stat_{S}$ such that $\seq{p_\alpha}{\alpha\in S_1}$ forms a delta system such that 
\begin{itemize}
\item $\forall \alpha \in S_1,\, p_\alpha\restrict \alpha = p_{\mathsf{fixed}}$ 
\item $\forall \delta \in S_1,\, \forall \alpha \in S_1 \cap \delta,\, \dom(p_\alpha) \subset \delta$
\end{itemize}
Note that, since $\alpha \notin \dom(p_\alpha)$, $c^{p_\alpha}$ can be decomposed into 
\begin{enumerate}
\item the fixed part $c^{p_\alpha \restrict \alpha } = c^{p_{\mathsf{fixed}}}$,
\item the finite part $c^{p_\alpha \restrict (\omega_1 \setminus \alpha)} \restrict \alpha $, and
\item the tail part $c^{p_\alpha }\restrict (\omega_1 \setminus \alpha) = c^{p_\alpha \restrict (\omega_1 \setminus \alpha)} \restrict (\omega_1 \setminus \alpha)$, which is above $\alpha$.
\end{enumerate}
Again, applying the pressing down lemma, pick $S_2 \in \stat_{S_1}$ such that $\forall \alpha \in S_2,\, c^{p_\alpha \restrict(\omega_1 \setminus \alpha)}\restrict \alpha  = c_{\mathsf{fixed}}$.
Then $\seq{p_\alpha}{\alpha \in S_2}$ is centered.
Indeed, for each $\Gamma \in [S_2]^{<\aleph_0}$, both
\[
\bigcup_{\alpha \in  \Gamma} p_\alpha = \bigcup_{\alpha \in  \Gamma} p_\alpha\restrict \alpha \cup p_\alpha\restrict (\omega_1\setminus \alpha) 
= p_{\mathsf{fixed}} \cup\bigcup_{\alpha \in  \Gamma}  p_\alpha\restrict (\omega_1\setminus \alpha) 
\]
and
\[
\bigcup_{\alpha \in  \Gamma} c^{p_\alpha} 
=\bigcup_{\alpha \in  \Gamma} c^{p_\alpha \restrict \alpha } \cup c^{p_\alpha \restrict (\omega_1 \setminus \alpha)} \restrict \alpha \cup c^{p_\alpha }\restrict (\omega_1 \setminus \alpha) 
= c^{p_{\mathsf{fixed}} } \cup c_{\mathsf{fixed}}  \cup \bigcup_{\alpha \in  \Gamma}  c^{p_\alpha \restrict(\omega_1 \setminus \alpha)}\restrict (\omega_1 \setminus \alpha)
\]
are functions.
\end{proof}
\begin{lem}\label{lem:E-uniformization}
$\ma_{\aleph_1}(\stat_E)$ implies $\mathsf{U}(\{\omega_1\setminus E\})$
\end{lem}
Conversly, a non-uniformizable ladder system coloring can be forced by the Cohen forcing.
\begin{lem}[\cite{MR1178369}]
Cohen forcing $\mathbb{C} = {}^{<\omega}2$ forces $\lnot\mathsf{U}(\stat)$.
\end{lem}

\begin{lem}\label{lem:stat non-unif pres}
If $\pair{\vec{l},\vec{c}}$ is a ladder system coloring that is not uniformizable on any stationary subset of $E$ and $\mdl{P}$ is $\stat_E$-Knaster forcing, then $\mdl{P}$ forces that $\pair{\vec{l},\vec{c}}$ is not uniformizable on any stationary subset of $E$.
\end{lem}
\begin{proof}
We prove the contrapositive.
Suppose that 
\begin{align}
p\forcestate[\mdl{P}]{$\d{S} \subset \c{E}$ is stationary and $\pair{\vec{l},\vec{c}}\c{}$ is uniformized over $\d{k}\colon \d{S}  \to \omega$}.\label{1}
\end{align}
Then $T = \{\alpha \in E \mid \exists q\leq p,\, q\forces{\c{\alpha} \in \d{S}} \}$ is stationary.
Fix a sequence $\pair{\bar{p},\bar{k}}=\seq{k_\alpha < \omega,\, p_\alpha \in \mdl{P}}{\alpha \in T}$ such that
\begin{align}
p_\alpha \forces[\mdl{P}] \alpha \in \d{S} \land \d{k}(\alpha) = \c{k}_\alpha.\label{2}
\end{align}
Pick a stationary set $T_1 \subset T$ such that $\bar{p}\restrict T_1$ is linked.
Then, for each $\alpha, \beta \in T_1$, \[(p_\alpha \land p_\beta \Vdash)\forall \xi \in l^{k_\alpha}_\alpha  \cap l^{k_\beta}_\beta,\, c_\alpha(\xi) = c_\beta(\xi).\]
Thus $\pair{\vec{l},\vec{c}}$ can be uniformized.
\end{proof}

\begin{thm}[$\ch$]\label{thm:local sep}
Some $\stat_E$-Knaster poset forces $\lnot \mathsf{U}(\stat_E) + \ma(\xknas{\stat_E}) + \lnot\ch$ and hence it forces $\lnot \mathsf{U}(\stat_E) + \mathsf{U}(\{\omega_1 \setminus E\})$ and $\lnot\ma(\xpc{\stat_{\omega_1 \setminus E}}) + \ma(\xknas{\stat_E})$.
\end{thm}
\begin{proof}
Cohen forcing $\mathbb{C}$ forces $\lnot\mathsf{U}(\stat) + \ch$.
In the Cohen extension $V^\mathbb{C}$, we force $\ma_{\aleph_1}(\xpc{\stat_E})$ by the standard bookeeping argument.
\end{proof}
The following is also immediate from \cref{lem:stat non-unif pres}.
\begin{lem}\label{lem:stat non-unif pres total}
$\stat$-Knaster posets preserve $\lnot\mathsf{U}(\stat)$
\end{lem}
Combininig \cref{lem:E-uniformization} and \cref{lem:stat non-unif pres total} we have the following.
\begin{thm}[$\ch$]\label{thm:global sep}
$\lnot\mathsf{U}(\stat) + \ma(\xknas{\stat}) + \lnot \ch$ is forced by a $\stat$-Knaster forcing, so is $\forall E \in \stat,\,\lnot\ma(\xpc{\stat_{\omega_1 \setminus E}}) + \ma(\xknas{\stat})$.
\end{thm}
\section{Stationary refinement v.s.~countable decomposition}
We now turn into the separation of forcing axioms for $\sigma$-linked and $\xpc{\stat}$.
A natural strategy is as follows:
\begin{itemize}
\item Construct a combinatorial structure such that every ``stationary part'' contains a ``nice stationary subpart'' (say, it has the SS property) but every countable decomposition contains a ``bad piece'' (say, it is not countably decomposable).
\item Any $\aleph_1$-$\sigma$-linked forcing preserves this situation for the constructed structure.
\item In contrast, a $\xpc{\stat}$ forcing destroys this situation, that is, every structure with SS property can be forced to be ``totally nice'' or decomposable into countably many ``nice" pieces.
\end{itemize}
We apply the above strategy to two structures: ladder system colorings and Aronszajn trees.
The corresponding notions of “nice,” the SS property, and countable decomposability for each structure are as following:
\begin{table}[hbt]
  \centering
  \begin{tabular}{|c|c|c|c|}
    \hline
    Structure & “Nice”  & SS property & Countable Decomposability \\
    \hline
    Ladder System Coloring & uniformized & SS-uniformizable & $\sigma$-uniformizable \\
    Aronszajn Tree & antichain & SS tree & special tree \\
    \hline
  \end{tabular}
  \label{tab:property-summary}
\end{table}

\subsection{SS uniformizablity and $\xpc{\stat}$}
\begin{dfn}
A ladder system coloring $\pair{\vec{l},\vec{c}}$ is \emph{SS-uniformizable} if, for every stationary set $E\subset\omega_1$, $\pair{\vec{l},\vec{c}}$ is uniformized on some  stationary set $E' \subset E$. 
$\pair{\vec{l},\vec{c}}$ is \emph{$\sigma$-uniformizable} if there exists a decomposition $\bigcup_{n < \omega}A_n = \omega_1$ such that, for each $n$, $\pair{\vec{l},\vec{c}}$ is uniformized on $A_n$.
\end{dfn}
\begin{lem}
Assume $\diamondsuit$. Then some $\sigma$-closed forcing forces $\forall E \in \stat,\, \mathsf{U}(\stat_E)$.
\end{lem}
\begin{proof}
We shall show that, for any ladder system $\pair{\vec{l},\vec{c}}$ and any stationary set $E \in \stat$, there exists a $\sigma$-closed forcing which forces that $\pair{\vec{l},\vec{c}}$ is uniformized on some stationary set $E' \subset E$.
By iterating such forcing posets with countable support, we obtain a forcing poset that witnesses the lemma. 
$\diamondsuit$ is used only for the bookkeeping argument in this process.

Let $\mdl{Q}_E(\vec{l},\vec{c})$ be the forcing of all pairs $p = \pair{e_p,u_p} \in {}^\alpha 2 \times {}^\alpha \omega$ ($\alpha < \omega_1$) such that \[\forall \gamma < \alpha,\, e_p(\gamma) = 1 \to (\gamma \in E \land c_\gamma =^* u_p \restrict l_\gamma).\]
The order of $\mdl{Q}_E(\vec{l},\vec{c})$ is the coordinatewise reverse inclusion.
It is obvious that $\mdl{Q}_E(\vec{l}, \vec{c})$ is $\sigma$-closed.
Thus, the only thing we have to show is that the working part $\bigcup_{p \in \d{G}}e_p^{\gets}(\{1\})$ of the generic uniformizing function $\bigcup_{p \in \d{G}}u_p$ is a stationary subset of $E$.

Let us assume that $\forcestate[\mdl{Q}_E(\vec{l},\vec{c})]{$\d{e} = \bigcup_{p \in \d{G}}e_p$ and $\d{E} = \d{e}^{\gets}(\{1\})$}$
and then our goal is to show the stationarity of $\d{E}$.
So we shall assume that $p \forcestate{$\d{C} \subset \omega_1$ is a club set}$.
Let $\seq{M_i}{i \leq \omega}$ be a countinuous $\subset$-chain of countable elementary submodels of $H_\theta$ ($\theta$ is large enough regular) that contain $\d{C}$, $p$, etc.
Let $\delta_i = \omega_1 \cap M_i$ for each $i \leq \omega$.
We shall find $q \leq p$ that forces $\delta_\omega\in \d{C} \cap \d{E}$.
To get such condition, we recursively construct decreasing sequence $\seq{p_n \in M_n \cap \mdl{Q}_E(\vec{l},\vec{c})}{n < \omega}$ of conditions such that each domain of $p_{n + 1}$ contains $\delta_n$ and $u_{n + 1}$ extends $c_{\delta_\omega} \restrict (l_{\delta_\omega} \cap \delta_{n+1} \setminus \delta_n)$.
For convenience, let $\pair{e_n, u_n}$ denote $p_n$.
Let $p_0 = p$.
Suppose that $p_n$ is defined.
Let $e_{n + 1} \in M_{n + 1}$ extends $e_n$ with the value zeros and $u_{n + 1} \in M_{n + 1}$ extends $u_n \cup c_{\delta_\omega} \restrict (l_{\delta_\omega} \cap \delta_{n+1} \setminus \delta_n)$.
Then $\pair{e_{n+1},u_{n + 1}} \in M_{n+1}$ is a condition.
Since $\d{C}$ is forced to be unbounded in $\omega_1$, we can select a condition $p_{n + 1} \leq \pair{e_{n+1},u_{n + 1}}$ in $M_{n + 1}$ and $\eta_{n + 1} \in \delta_{n + 1}\setminus \delta_n$ such that $p_{n + 1} \forces{\c{\eta}_{n + 1} \in \d{C}}$.
Now $\seq{p_n}{n < \omega}$ is constructed and then $p_\omega = \pair{\bigcup_{n < \omega}e_n \cup \{\pair{\delta,1}\}, \bigcup_{n < \omega}u_n \cup \{\pair{\delta,0}\}}$ is a required condition.
\end{proof}
Sakai taught me a proof of the following lemma.
\begin{lem}
$\diamondsuit$ implies $\lnot\sigma\mathsf{U}$.
\end{lem} 
We will prove a stronger lemma in \cref{sec:4.3}, so we omit the proof here.
Thus we get an SS-uniformizable ladder system coloring which can not be decompose into countable pieces such that the coloring is non-uniformizable on each piece.

Preservation by $\aleph_1$-$\sigma$-linked forcings also runs as follows.
\begin{lem}\label{lem:SS pres}
Any $\xknas{\stat}$ forcing preserves SS uniformizable ladder system colorings in the ground model.
\end{lem}
We note that, since any $\aleph_1$-$\sigma$-linked poset is $\xknas{\stat}$, it preserves SS uniformizable ladder system coloring too.
\begin{proof}
Let $\pair{\vec{l},\vec{c}}$ be an SS uniformizable ladder system coloring and $\mdl{P}$ be an $\aleph_1$-$\sigma$-linked forcing.
We assume that $p \forces[]{\d{E} \in \stat}$.
Let $E_0 = \{\alpha < \omega_1 \mid \exists q \leq p,\, q \forcestate[]{$\alpha \in \d{E}$}\}$.
Then $E_0$ is stationary. Fix a sequence $\seq{q_\alpha \leq p}{\alpha \in E_0}$ such that $q_\alpha \forcestate[]{$\c{\alpha} \in \d{E}$}$.
Select a statinoary set $E_1 \subset E_0$ and $f \colon E_1 \to \omega$ such that $\bigcup_{\alpha \in E_1} c_\alpha \restrict l_\alpha^{f(\alpha)}$ is a function.
Since $\mdl{P}$ is ccc, there exists $q^* \leq p$ such that 
\[q^* \forcestate{$\d{E}' = \{\alpha \in \c{E}_1 \mid q_\alpha \in \d{G}_{\mdl{P}}\}$ is a stationary subset of $\d{E}$}.\]
The stationarity is the same as in the proof of \cref{lem:statpc pres}.
For $q^* \forces{\d{E}' \subset \d{E}}$, fix any $q\leq q^*$ and $q \forces{\c{\alpha} \in \d{E}'}$.
Then $q$ and $q_\alpha$ are compatible and their common extension forces $\c{\alpha} \in \d{E}$.
It is clear that $q^* \forcestate[]{$\c{f}$ uniformizes $\pair{\vec{l},\vec{c}}$ on $\d{E}'$}$.
\end{proof}
\begin{lem}\label{lem:nonsigma pres0}
Let $\pair{\vec{l},\vec{c}}$ be a not $\sigma$-uniformizable ladder system coloring.
Any $\aleph_1$-$\sigma$-linked forcing forces that $\pair{\vec{l},\vec{c}}$ is not $\sigma$-uniformizable.
\end{lem}
This lemma is a special case of \cref{lem:nonsigma pres} ($T = \omega_1$) in \cref{sec:4.3}.
So we omit the proof here too.

Thus the preservation step is also clear.
However, the following remains open.
\begin{question}
Does $\ma_{\aleph_1}(\xpc{\stat})$ imply that every SS-uniformizable ladder system coloring is $\sigma$-uniformizable?
\end{question}
\subsection{SS Aronszajn trees and $\xpc{\stat}$}
For a $\omega_1$-tree $\pair{T,<_T}$, $\alpha < \omega_1$, and $t,s \in T$,
\begin{itemize}
\item $\height(t)$ is the height of $t$, that is, the order type of $\{s \in T \mid s <_T t\}$,
\item $T_\alpha = \{t \in T \mid \height(t) = \alpha \}$ and $T\restrict X = \{t \in T \mid \height(t) \in X \}$ for $X \subset \omega_1$,
\item $t \restrict \alpha$ denotes the element of $T_{\alpha}$ below $t$,
\item $T$ is Aronszajn if all $T_\alpha$ and chains in $T$ are countable,
\item $T$ is Suslin if all antihcains and chains in $T$ are countable, and
\item $T$ is specialized Aronszajn tree if $T$ is Aronszajn tree that is the union of countably many antichains in $T$.
\end{itemize}
Hanazawa \cite{c61c0038-25ea-3bcf-bd99-5a5572293ef3} defined the following property.
\begin{dfn}
An Aronszajn tree $T$ has the property \emph{UU (SS)} if every uncountable (stationary) subset of $T$ contains an uncountable (statioanry) antichain.
\end{dfn}
Note that any special Aronszajn tree is SS by the pressing down lemma.

A Suslin tree $S$ is an object living in the constructible universe $L$.
To destroy a Suslin tree by forcing method, two strategies can be considered: adding an uncountable chain and adding an uncountable antichain.
The former can be forced by the tree $S$ (ordered by reversed tree order) itself.
The later can be forced by the poset $\finantich(S)$ of all finite non-maximal antichains in the tree $S$ ordered by reverse inclusion.
To add antichains more strongly, the following poset works.
\begin{dfn}[Baumgartner-Malitz-Reinhardt \cite{1361699995564696064}]
Let $T$ be an Aronszajn tree.
Then the poset $\specialize(T):= \{p\in \mathrm{Fn}(T, \omega) \mid \forall n<\omega,\, \text{$p^\gets(\{n\})$ forms an antichain}\}$ ordered by reverse inclusion is called the \emph{specializer of $T$}.
\end{dfn}
\begin{lem}\label{lem:SS tree antichain refining}
Fix a stationary set $E \subset \omega_1$ and an SS Aronszajn tree $T$.
Let $\seq{x_\alpha}{\alpha \in E}$ be a sequence of finite antichains of $T$.
Then there exists a stationary set $E' \subset E$ such that $\bigcup_{\alpha \in E'} x_\alpha$ is an antichain.
\end{lem}
The following proof is a straightforward improvement of the corresponding lemma in \cite{Steprans1988}.
\begin{proof}
Refining $E$ to a stationary set, we may assume that all $x_\alpha \cap T\restrict \alpha$ are equal to each other.
So let us assume that $x_\alpha = \{t_{\alpha,i} \mid i < n_\alpha\}$ and $\height(t_{\alpha,i}) \geq \alpha$.
If $\{t_{\alpha,i} \restrict \alpha \mid \alpha \in E,\, i < n_\alpha\}$ is an antichain, so is $\{t_{\alpha,i} \mid \alpha\in E ,\, i< n_\alpha\}$.
Thus we may assume that $\height(t_{\alpha,i}) = \alpha$ for each $\alpha$.
Furthermore, refining $E$ again, we may assume that $n_\alpha = n$ (fixed).
By SS property of $T$, we can assume that, for each $i < n$, $\{t_{\alpha,i} \mid \alpha \in E\}$ is an antichain.

By induction on $n$, we shall show that, for each $E' \in \stat_E$, there exists $E'' \in \stat_{E'}$ such that $\{t_{\alpha,i} \mid \alpha \in E'',\, i < n\}$ is an antichain.

\begin{description}
\item[The case $n = 1$] It is clear from the SS property of $T$.
\item[The case $n = 2$] Let $E' \in \stat_E$ and, for $\alpha \in E'$,  \[X_\alpha = \{\alpha \in E' \mid t_{\alpha,0} < t_{\beta,1}\, \lor\, t_{\alpha,1} < t_{\beta,0} \}.\] We consider the two cases:
 \begin{description}
 \item[$X_\alpha$ is stationary for an $\alpha \in E'$] Suppose that $X_{\alpha,0} = \{\beta \in E' \mid t_{\alpha,0} < t_{\beta,1}\}$ is stationary.
 Suppose that $\beta < \beta'$ are members of $X_{\alpha,0}$ such that $x_\beta \cup x_{\beta'}$ is not an antichain.
 If $t_{\beta,0} < t_{\beta',1}$, then $t_{\alpha,0}$ and $t_{\beta,0}$ are compatible since both of them are below $t_{\beta',1}$.
 Otherwise, $t_{\alpha,0} < t_{\beta,1} < t_{\beta',0}$, it is impossible.
 We obtain contradictions in both cases.
 Thus, $\bigcup_{\beta \in X_{\alpha,0}} x_\beta$ is an antichain.
 \item[$X_\alpha$ is non-stationary for any $\alpha \in E'$]
 $Y = \Delta_{\alpha \in E'}E' \setminus X_\alpha$ is a stationary set (in fact, club in $E'$).
Fix any $\alpha < \beta$ in $Y$.
Then $\beta \notin X_\alpha$.
Thus $t_{\alpha,0}$ and $t_{\beta,1}$ are incompatible, so are $t_{\alpha,1}$ and $t_{\beta,0}$.
This implies $x_\alpha \cup x_\beta$ is an antichain.
Thus $\bigcup_{\alpha \in Y}x_\alpha$ is an antichain.
 \end{description}
\item[The caes $n > 2$] By induction hypothesis, we have stationary sets $E' \supset E(0) \supset E(1) \supset E(2)$ such that all of $\{t_{\alpha,i} \mid i < n -1,\,\alpha \in E(0)\}$, $\{t_{\alpha,i} \mid 0 < i < n,\,\alpha \in E(1)\}$, and $\{t_{\alpha,i} \mid i \in \{0,n-1\},\,\alpha \in E(2)\}$ are antichains.
Then, $\bigcup_{\alpha \in E(2)}x_\alpha$ is an antichain.
\end{description}
\end{proof}
The correspondence between the UU property of trees and the precaliber $\aleph_1$ property of their specializer was established by Steprans and Watson.
\begin{lem}\cite{Steprans1988}
For an Aronszajn tree $T$, the folowing are equivalent:
\begin{itemize}
\item $T$ is UU.
\item the poset $\finantich(T)$ of all finite antichain is precaliber $\aleph_1$.
\end{itemize}
\end{lem}
The analogue of this lemma for stationary sets immediate from \cref{lem:SS tree antichain refining}.
\begin{lem}
For an Aronszajn tree $T$, the following are equivalent:
\begin{condenum}
\item $T$ is SS.
\item $\finantich(T)$ is $\xpc{\stat}$.
\item $\specialize(T)$ is $\xpc{\stat}$
\end{condenum}
\end{lem}
\begin{lem}\label{al1 sigl pres nonsp}
$\aleph_1$-$\sigma$-linked posets preserve the non-speciality of Aronszajn trees. 
\end{lem}
\begin{proof}
Let $\mdl{P}$ be a $\aleph_1$-$\sigl$ poset and $T = \omega_1$ be an Aronszajn tree.
Suppose that \[p \forcestate{$T$ is special}.\]
Then we have a sequence $\seq{\d{A}_n \in V^\mdl{P}}{n < \omega}$ of names such that \[p \forcestate{$T = \bigcup_{n < \omega}\d{A}_n$ and $\d{A}_n$ is an antichain in $T$}.\]
Let $\seq{p_\alpha \leq p,\,n_\alpha < \omega}{\alpha < \omega_1}$ be a sequence such that
\[p_\alpha \forces{\c{\alpha} \in \d{A}_{n_\alpha}}\]
Since $\mdl{P}$ is $\aleph_1$-$\sigl$, there exists a decomposition $\{p_\alpha \mid \alpha \in \omega_1\} = \bigcup_{n < \omega}X_n$ where each $X_n$ is linked.
Let $B_{m,n} = \{\alpha \in T \mid n = n_\alpha\, \land\,p_\alpha \in X_m  \}$.
Then, $T = \bigcup_{m,n}B_{m,n}$ and each $B_{m,n}$ is an antichain.
\end{proof}

\begin{lem}\label{statk pres sarp}
$\stat$-(K) posets preserve SS property of Aronszajn trees.
\end{lem}
\begin{proof}
Suppose that $T$ is an SS Aronszajn tree.
Fix any $\d{S}\in V^\mdl{P}$ and $p \in \mdl{P}$ such that $p \forces{\d{S} \in \stat(\c{T}) }$.
Let $p\forces{\d{E}\in\stat \land \{\d{t}_\alpha \in T_{\alpha} \mid \alpha \in \d{E}\} \subset \d{S}}$.
Let $E_0 = \{\alpha \in \omega_1 \mid \exists q \leq p,\,\exists t \in T_{\alpha},\, q \forces{\alpha \in \d{E} \land \c{t} =\d{t}_\alpha}\}$.
Then $E_0$ is stationary.
Fix $\seq{q_\alpha \leq p,\,t_\alpha \in T}{\alpha \in E_0}$ such that 
\begin{align}
q_\alpha \forces{\c{\alpha}\in\d{E} \land \c{t}_\alpha = \d{t}_\alpha}.\label{a}
\end{align}
Select a stationary set $E_0' \subset E_0$ such that $\{t_\alpha \mid \alpha \in E'_0\}$ is an antichain.
Since $\mdl{P}$ is ccc, there exists $q^* \leq p$ such that
\[q^* \forcestate{$\d{E}' = \{\alpha \in \c{E}'_0 \mid q_\alpha \in \d{G}_{\mdl{P}}\}$ is a stationary subset of $\d{E}$}\]
The stationarity is the same as in the proof of \cref{lem:statpc pres}. 
For $q^* \forces{\d{E}' \subset \d{E}}$, fix any $q\leq q^*$ and $q \forces{\c{\alpha} \in \d{E}'}$.
Then $q$ and $q_\alpha$ are compatible and their common extension forces $\c{\alpha} \in \d{E}$.
Thus we obtain \cref{a}.
\begin{description}
\item[Claim]$q^* \forcestate{$\{\d{t}_\alpha \mid \alpha \in \d{E}'\}$ is a stationary antichain}$.
	\begin{description}
	\item[proof]Stationarity is clear.
	Suppose that $q\leq q^*$, $\alpha,\beta \in E'_0$ and \[q \forcestate{$\alpha,\beta \in \d{E}'$ and $\d{t}_\alpha$ and $\d{t}_\beta$ are compatible}.\]
	Then $q\forces{\exists r \in \d{G}_\mdl{P},\, r \leq \c{q}_\alpha,\c{q}_\beta}$.
	Thus there exists $r \leq q_\alpha,q_\beta$ such that $r \forcestate{$\d{t}_\alpha = \c{t}_\alpha$ and $\d{t}_\beta = \c{t}_\beta$ are compatible}$.
	This contradicts the choice of $E'_0$.
	\end{description}
\end{description}
\end{proof}
Thus, we have the following:
\begin{itemize}
\item $\ma_{\aleph_1}(\xpc{\stat})$ implies that every SS Aronszajn tree is special.
\item Every $\aleph_1$-$\sigma$-linked forcing preserves every SS but non-specail Aronszajn trees.
\end{itemize}
In particular, we have the following.
\begin{prop}
If a Hanazawa tree exists and $2^{\aleph_1} = \aleph_2$, then an $\aleph_1$-$\sigma$-linked forcing forces $\ma_{\aleph_1}(\sigl) + \lnot \ma_{\aleph_1}(\xpc{\stat})$.
\end{prop}
However, the existence of an SS and non-special Aronszajn tree remains open.
With the reader's indulgence, we propose the following terminology.
\begin{dfn}
A \emph{Hanazawa tree} $T$ is an SS Aronszajn tree that is not special.
\end{dfn}
\begin{question}[Hanazawa \cite{c61c0038-25ea-3bcf-bd99-5a5572293ef3}]
Does a Hanazawa tree exists?
\end{question}
\subsection{Uniformization on trees and the separation of $\ma_{\aleph_1}(\sigl)$ from $\ma_{\aleph_1}(\xpc{\stat})$}\label{sec:4.3}
In this subsection, we give a proof of the separation $\ma_{\aleph_1}(\sigl)$ from $\ma_{\aleph_1}(\xpc{\stat})$.
\begin{thm}\label{thm:distinguish sigl and statpc}
$\ma_{\aleph_1}(\sigl)$ does not imply $\ma_{\aleph_1}(\xpc{\stat})$.
\end{thm}
The following is a variant of the $T$-uniformiation (\cite{MooreTunif}).
\begin{dfn} 
Let $T$ be an $\aleph_1$-tree and $\pair{\vec{l},\vec{c}}$ be a ladder system.
Then, a function $u \colon S \to \omega$ on a subtree $S$ of $T$ is a \emph{$T$-uniformization} if $\forall \alpha \in \omega_1 \cap \limo,\, \forall s \in S_\alpha,\, \forall^\infty \gamma \in l_\alpha,\, u(s \restrict \gamma) = c_\alpha(\gamma)$.
In case $S = T$, such $u$ is called a \emph{full-$T$-uniformization}.
If a (full-)$T$-uniformization exists, then we say that the ladder system coloring is \emph{(full-)$T$-uniformizable}.
For a parical function $d\colon \omega_1 \to \omega$, let $d^T \colon T \to \omega$ such that $d^T(t) = d(\height(t))$
We say $\pair{\vec{l},\vec{c}}$ is \emph{$\sigma$-$T$-uniformizable} if, there exists a countable decomposition $\bigcup_{n < \omega}A_n = T \restrict \limo$ and $f \colon T\restrict \limo \to \omega$ such that $\bigcup_{t \in A_n}c^T_\alpha\restrict [t\restrict l_{\height(t),f(t)},t)$ is a function.
$\sigma\mathsf{U}_T$ is the assertion that every ladder system coloring is $\sigma$-$T$-uniformizable.
\end{dfn}

\begin{lem}
Let $T$ be an Aronszajn tree.
Then $\diamondsuit$ implies $\lnot\sigma\mathsf{U}_T$.
\end{lem}
The following proof follows immediately from the one of $\diamondsuit \to\lnot \sigma\mathsf{U}$ that the author was learn from Sakai.
\begin{proof}
Towards a contradiction, we assume both of $\diamondsuit$ and $\sigma\mathsf{U}_T$.
By $\diamondsuit$, there exists a sequecne $\vec{d} = \seq{d_\alpha \colon T\restrict\alpha \to 2}{\alpha \in \omega_1 \cap \limo}$ such that, for each $d \colon T \to 2$, $\{\alpha < \omega_1 \mid d\restrict (T\restrict \alpha) = d_\alpha\}$ is stationary.

Let $\bigcup_{n < \omega}X_n = \omega_1$ be a decomposition to countable many uncountable sets.
Then $C = \{\alpha < \omega_1 \mid \forall n,\, \sup(X_n \cap \alpha) = \alpha\}$ is a club set.
Let $\vec{l}$ be a ladder system such that $\sup(l_\alpha \cap X_n) = \alpha$ for each $\alpha \in C$ and $n < \omega$.
It is easy to construct a coloring $\vec{c}$ of $\vec{l}$ such that 
\begin{align}
\forall \alpha \in C,\,\forall n < \omega,\,\forall t\in T_\alpha,\,\exists^\infty \gamma \in l_\alpha \cap X_n,\, c_\alpha(\gamma) \neq d_\alpha(t\restrict \gamma).
\label{assertion:diamond5}
\end{align}
By $\sigma\mathsf{U}_T$, there exists a decomposition $\bigcup_{n < \omega}A_n = T\restrict\limo$ and functions $u_n \colon T \to 2$ such that 
\begin{align}
\forall n,\, \forall t \in A_n,\, \forall^\infty \gamma \in l_{\height(t)},\, u_n(t\restrict \gamma) = c_\alpha(\gamma).
\label{assertion:diamond6}
\end{align}

We are now ready to derive a contradiction.
Let $u = \bigcup_{n < \omega}u_n \restrict (T\restrict X_n)$.
Since $\vec{d}$ is a  $\diamondsuit$-sequence, there exists $\alpha \in C$ such that $u \restrict (T\restrict \alpha) = d_\alpha$.
Take $n < \omega$ and $t \in A_n$ such that $\height(t) = \alpha$.
Then by \cref{assertion:diamond5,assertion:diamond6}, there exists $\gamma \in l_\alpha \cap X_n$ such that $d_\alpha(t\restrict \gamma) \neq c_\alpha(\gamma) = u_n(t\restrict\gamma)$.
Since $\gamma \in X_n$, $u_n(t\restrict \gamma) = u(t \restrict \gamma)$ and hence $u(t \restrict \gamma) \neq d_\alpha(t \restrict \gamma)$, it contradicts to $u \restrict (T\restrict \alpha) = d_\alpha$
\end{proof}

Since any countable support iteration of $\sigma$-closed forcings preserves $\diamondsuit$, we have the following.
\begin{thm}\label{thm:diamond + stat unif}
$\diamondsuit + \forall E \in \stat,\,\mathsf{U}(\stat_E)$ is consistent relative to $\zfc$, so is $\lnot\sigma\mathsf{U}_T + \forall E \in \stat,\,\mathsf{U}(\stat_E)$.
\end{thm}

\begin{lem}\label{lem:nonsigma pres}
Let $T$ be an Aronszajn tree.
Suppose that $\pair{\vec{l},\vec{c}}$ is not $\sigma$-$T$-uniformizable.
Then any $\aleph_1$-$\sigma$-linked forcing forces that $\pair{\vec{l},\vec{c}}$ is not $\sigma$-$T$-uniformizable.
\end{lem}
\begin{proof}
Let $\mdl{P}$ be an $\aleph_1$-$\sigma$-linked forcing and $\pair{\vec{l},\vec{c}}$ be a non-$\sigma$-$T$-uniformizableladder system coloring.
Towards a contradiction, we assume that $p \in \mdl{P}$ forces that
$T\restrict \limo = \bigcup_{n < \omega}\d{A}_n$ and $\d{f} \colon T\restrict \limo \to \omega$ satisfies $\forall n < \omega,\,\{\pair{t\restrict l_{\alpha,k} ,c_\alpha(l_{\alpha,k})} \mid t \in \d{A}_n,\,\height(t) = \alpha,\,\d{f}(t)\leq k\} = \bigcup_{t \in \d{A}_n} c^T_{\height(t)}\restrict [t\restrict l_{\height(t),\d{f}(t)},t)$ is a function.
Fix a sequence $\seq{p_t \leq p,\, n_t < \omega,\,k_t < \omega}{t \in T}$ such that 
\[p_t\forcestate[\mdl{P}]{$t \in \d{A}_{n_t}$ and $\d{f}(t) = \c{k}_t$}.\]

Since $\mdl{P}$ is $\aleph_1$-$\sigma$-linked and $|T| = \aleph_1$ , there exists a cover $\bigcup_{m < \omega}K_m = T$ such that each $\{p_t \mid t \in K_m\}$ is linked.
Let $B_{m,n} = \{t \in T \mid m_t = m,\,n_t = n\}$.

Since $\pair{\vec{l},\vec{c}}$ is not $\sigma$-$T$-unoformizable, there exists $m,n,k < \omega$ such that $\pair{\vec{l},\vec{c}}$ is not uniformized on $B_{n,m}$.
In particular, the function $t\mapsto k_t$ does not uniformize $\pair{\vec{l},\vec{c}}$.
Thus there exists $s,t \in B_{n,m}$ and $k \geq \max(k_s,k_t)$ such that $l_{\alpha,k} = l_{\beta,k} = \gamma$, $s\restrict \gamma = t \restrict \gamma$ and $c_\alpha(\gamma) \neq c_\beta(\gamma)$.
On the other hand, a common extension of $p_t$ and $p_s$ forces not, it is a contradiction.
\end{proof}

\begin{lem}
Assume that $\ma_{\aleph_1}(\xpc{\stat})$ holds and $T$ is an SS Aronszajn tree.
Then every SS-uniformizable ladder system coloring $\pair{\vec{l},\vec{c}}$ is full-$T$-uniformized.
\end{lem}
\begin{proof}
Let $\mathbb{P}_T(\vec{l},\vec{c})$ be the forcing poset of all finite partial functions $p \colon T\restrict \limo \to \omega$ such that $c^p := \bigcup_{t \in \dom (p)} c^T_\alpha \restrict [t \restrict l_{\alpha,p(t)}, t)$ 
  is a function.
The order is defined by reverse inclusion.
It is clear that this forcing provides a full-$T$-uniformizing function.

We shall show that $\mdl{P}_T(\vec{l},\vec{c})$ is $\xpc{\stat}$.
So let $E \subset \omega_1$ be a stationary set and $\seq{p_\alpha }{\alpha \in E}$ an $E$-indexed sequence of conditions.
Let $p^-_\alpha = p_\alpha \restrict (T\restrict \alpha)$, $p^0_\alpha = p_\alpha \restrict (T_\alpha)$, and $p^+_\alpha = p_\alpha \restrict (T\restrict (\omega_1 \setminus \alpha))$.
Furthermore, $x^*_\alpha = \dom(p^*_\alpha)$ for $* \in \{-,0,+\}$ and $x_\alpha = x^-_\alpha \cup x^0_\alpha \cup x^+_\alpha$.
We may assume that
\begin{itemize}
\item each $p^-_\alpha$ is fixed to $p^-$
\item $\forall \xi \in E\cap \alpha,\, x_\xi \subset T\restrict \alpha$.
\item $\max\{p_\alpha(t)\mid t\in x_\alpha\}$ is fixed to $m_0$,
\item $\max\{l_{\height(t)}(m_0) \mid t \in x_\alpha \}$ is fixed to $\delta_0$,
\item $\bigcup_{\alpha \in E}c_\alpha\restrict (l_\alpha \setminus \gamma_f)$ is a function where $\gamma_f = l_{\alpha,n_f}$(using the SS property of $\pair{\vec{l},\vec{c}}$)
\item $c^{p^0_\alpha}\restrict (T \restrict (\gamma_f \cap l_\alpha))$ is fixed to $d$,
\begin{itemize}
\item at this step, we already have $\{p^-_\alpha \cup p^0_\alpha \mid \alpha \in E\}$ is centered.
\end{itemize}
\item $\bigcup_{t \in x^+_\alpha}c_{\height(t)} \restrict \alpha$ is fixed to $c^+$ and $\alpha_c = \sup (\dom(c))$,
\item $\bigcup_{\alpha \in E}(x^0_\alpha \cup \{t\restrict \alpha \mid t \in x^+_\alpha \})$ is an antichain by \cref{lem:SS tree antichain refining}.
\end{itemize} 

Then, for a finite set $\Gamma \subset E$, we have 
\begin{align*}
\bigcup_{\alpha \in \Gamma} c^{p_\alpha} &= \bigcup_{\alpha \in \Gamma} c^{p^-_\alpha}  \cup c^{p^0_\alpha} \cup c^{p^+_\alpha} 
\\&= c^{p^-} \cup \bigcup_{\alpha \in \Gamma} c^{p^0_\alpha} \restrict( T\restrict (l_\alpha \setminus \gamma_f)) \cup c^{p^0_\alpha} \restrict(T \restrict  (l_\alpha \cap \gamma_f)) \cup c^{p^+_\alpha} \restrict \alpha \cup c^{p^+_\alpha} \restrict (\omega_1 \setminus \alpha)\\
&= c^{p^-} \cup \bigcup_{\alpha \in \Gamma} c^{p^0_\alpha} \restrict( T\restrict (l_\alpha \setminus \gamma_f)) \cup d \cup c^+\cup c^{p^+_\alpha} \restrict (T\restrict (\omega_1 \setminus \alpha) )\\
&= c^{p^-} \cup d \cup c^+ \cup \bigcup_{\alpha \in \Gamma} c^{p^0_\alpha} \restrict( T\restrict (l_\alpha \setminus \gamma_f)) \cup c^{p^+_\alpha} \restrict (T\restrict (\omega_1 \setminus \alpha) )
\end{align*}
is a function:
For members $\pair{t,n}$ and $\pair{t,m}$ of the above set, We consider the following three cases.
\begin{description}
\item[Both $\pair{t,n}$ and $\pair{t,m}$ come from same $\alpha$:] 
Ofcourse $n = m$ because $p_\alpha$ is a condition.
\item[$\pair{t,n} \in c^{p^0_\alpha} \restrict (T\restrict (l_\alpha \setminus \gamma_f))$ and $\pair{t,m} \in c^{p^0_\beta} \restrict  (T\restrict (l_\beta \setminus \gamma_f))$:] 
Since $\bigcup_{\alpha \in E} c_\alpha \restrict (l_\alpha \setminus \gamma_f)$ is a function, $m = n$
\item[$\pair{t,n} \in c^{p^+_\alpha} \restrict (T\restrict (\omega_1 \setminus \alpha) )$ and $\pair{t,m}$ comes from a piece indexed by other $\beta$:] 
This case does not occurs since $\bigcup_{\xi \in E}x^0_\xi \cup  \{t\restrict \xi \mid t \in x^+_\xi \}$ is an antichain.
\end{description}
The above cases are collectively exhaustive.


\end{proof}

We are now ready to derive the separation of $\ma_{\aleph_1}(\sigl)$ from $\ma_{\aleph_1}(\xpc{\stat})$.
\begin{proof}[Proof of \cref{thm:distinguish sigl and statpc}: $\ma_{\aleph_1}(\sigl)$ does not imply $\ma_{\aleph_1}(\xpc{\stat})$]
By \cref{thm:diamond + stat unif}, we have a model of $\lnot\sigma\mathsf{U}_T + \forall E \in \stat,\,\mathsf{U}(\stat_E)$.
Iterating $\aleph_1$-$\sigma$-linked forcings, we can force $\ma_{\aleph_1}(\sigl)$ and $\lnot\sigma\mathsf{U}_T$ simulatenously.
Furthermore, by \cref{lem:nonsigma pres,lem:SS pres}, any SS uniformizable  ladder system coloring that witnesses $\lnot \sigma\mathsf{U}_T$ in the ground model remains a ladder system coloring with the same propertyin the extension model.

On the other hand, $\ma_{\aleph_1}(\xpc{\stat})$ implies that any SS uniformizable ladder system coloring is full-$T$-uniformizable for any SS Aronszajn tree $T$ (such $T$ always exists, since any spceial Aronszajn tree is SS).
Thus any ladder system coloring in the ground model witnesses the failure of $\ma_{\aleph_1}(\xpc{\stat})$.
\end{proof}

\bibliographystyle{plain} 
\bibliography{bib} 

\end{document}